\documentclass[runningheads,a4paper]{llncs}

\usepackage{amssymb}
\usepackage{graphicx}
\usepackage{color}
\usepackage{url}
\urldef{\mailsa}\path|buesing@or.rwth-aachen.de|
\urldef{\mailsb}\path|d.andreagiovanni@zib.de|
\newcommand{\keywords}[1]{\par\addvspace\baselineskip
\noindent\keywordname\enspace\ignorespaces#1}

\newcommand{\SMB}{\mathcal{S}_M}

\usepackage{algorithm2e}
\usepackage{float}
\floatstyle{ruled}
\newfloat{algo}{tbp}{loa}[chapter]
\providecommand{\algoname}{Algorithm}
\floatname{algo}{\protect\algoname}
\newenvironment{lyxlist}[1]
{\begin{list}{}
{\settowidth{\labelwidth}{#1}
 \setlength{\leftmargin}{\labelwidth}
 \addtolength{\leftmargin}{\labelsep}
 }}
{\end{list}}

\begin{document}

\mainmatter  

\title{Robust Optimization
\\[2pt]
under Multi-band Uncertainty
\\[6pt]
Part I: Theory
\thanks{This work was partially supported by the \emph{German Federal Ministry of Education and Research} (BMBF), project \emph{ROBUKOM} \cite{BlDAHa11,KoEtAl12}, grant 03MS616E, and by the DFG Research Center \textsc{Matheon} - \textit{Mathematics for key technologies}, project B23  ``\emph{Robust optimization for network applications}''.
}
}

\titlerunning{Robust Optimization under Multi-band Uncertainty - Part I: Theory}

%
%
\author{Christina B\"using\inst{1} \and Fabio D'Andreagiovanni\inst{2,3}
}
\authorrunning{C. B\"using and F. D'Andreagiovanni}

\institute{Chair of Operations Research, RWTH Aachen University\\
Kackertstrasse 7, 52072 Aachen, Germany\\
\mailsa\\
\
\\
\and
Department of Optimization, Zuse-Institut Berlin (ZIB)\\
Takustrasse 7, 14195 Berlin, Germany\\
\and
Department of Computer, Control and Management Engineering\\
Sapienza Universit\`{a} di Roma, via Ariosto 25, 00185 Roma, Italy\\
\mailsb
}

%
%

\toctitle{Lecture Notes in Computer Science}
\tocauthor{Authors' Instructions}
\maketitle

\begin{abstract}
The classical single-band uncertainty model introduced by Bertsimas and Sim \cite{BeSi04} has represented a breakthrough in the development of tractable robust counterparts of Linear Programs. However, adopting a single deviation band may be too limitative in practice: in many real-world problems, observed deviations indeed present asymmetric distributions over asymmetric ranges, so that getting a higher modeling resolution by partitioning the band into multiple sub-bands is advisable.
The critical aim of our work is to close the knowledge
gap on the adoption of multi-band uncertainty in Robust Optimization:
a general definition and intensive theoretical study of a multi-band
model are actually still missing. Our new developments have been
also strongly inspired and encouraged by our industrial partners, interested in getting a better modeling of arbitrary shaped distributions, built on historical data about the uncertainty affecting the considered real-world problems.

\end{abstract}


\section{Introduction} \label{sec:intro}

A fundamental assumption in classical optimization is that all data are exact.
However, many real-world problems involve data that are uncertain or not known
with precision. This can be due, among other things, to erroneous measurements or adoptions of approximated
numerical representations. It is well-known that neglecting such uncertainty may have very bad effects: solutions that are feasible may reveal to be infeasible, as a result of small deviations from nominal data, while solutions considered optimal may turn out to be of very low quality.  Such risks cannot be taken in critical applications, such as contaminant detection in water distribution networks \cite{WaHaMu06} or the design of power grids \cite{JiEtAl12}.
Since the groundbreaking investigations by Dantzig \cite{Da55}, many works have thus tried to find effective ways to deal with uncertainty (see \cite{BeBrCa11} for an overview).

In recent years, Robust Optimization (RO) has become a valid methodology to tackle uncertainty in optimization and has increasingly attracted the attention of practitioners, that have started many research partnerships on the topic with academic institutions (e.g., on optical network design \cite{BeKoNo11}, telecommunications network design \cite{KoEtAl12}, surgery scheduling \cite{MaNiNo12}).
The central feature of RO is to approach coefficient uncertainties through hard constraints, that restrict the feasible set maintaining only \emph{robust} solutions, i.e. solutions protected against deviations. For an exhaustive introduction to theory and applications of RO, we refer the reader to the book by Ben-Tal et al. \cite{BeElNe09} and to the recent survey by Bertsimas et al. \cite{BeBrCa11}.

One of the most successful and widely-adopted RO approach is the so-called $\Gamma$-scenario set (BS), introduced by Bertsimas and Sim in \cite{BeSi04}. The uncertainty model of BS
assumes that each uncertain coefficient $a$ is a symmetric and bounded random variable, taking value in the symmetric interval $[\bar{a}-d,\bar{a}+d]$, where $\bar{a}$ is the 
\emph{nominal value} of $a$ and $d$ is the maximum deviation from $\bar{a}$.
Furthermore, the uncertainty set adopts a parameter $\Gamma > 0$  to represent
the maximum number of coefficients that are allowed to deviate from their nominal value.
%
These modeling assumptions lead to a compact RO model that preserves the features of the original nominal model.
However, the use of a single deviation band may greatly limit the power of modeling
uncertainty, as noted, for example, even by Sim and colleagues in \cite{ChSiSu07}. This is particularly evident in real-world problems, where it is common to have \emph{asymmetric probability distributions} of the deviations, that are additionally defined over \emph{non-symmetric} intervals.
In such cases, neglecting the inner-band behavior and
just considering the extreme values like in BS may lead to a rough estimate of the
deviations and thus to unrealistic uncertainty set, which either overestimate or
underestimate the overall deviation. Having a higher modeling resolution would
therefore be very desirable. This can be accomplished by breaking the single
band into multiple and narrower bands, each with its own $\Gamma$ value.
This observation was first captured by Bienstock and taken into account
to develop a RO framework for the special case of  Portfolio Optimization \cite{Bi07}. It was then
extended to Wireless Network Design \cite{BiDA09,DA11}. Yet, no definition and intensive theoretical
study of a more general multi-band model applicable in other contexts
have been accomplished. The main goal of this paper is to close such a gap.

We remark that investigating the adoption of a multi-band uncertainty set
and studying the theoretical properties of the resulting model have been discussed
with and strongly encouraged by our industrial partners, such as \emph{British
Telecom Italia} (BT) and \emph{Nokia Siemens Networks} (NSN) \cite{NSN12}, in past and present
collaborations about real network design \cite{BiDA09,BlDAHa11,DA11,KoEtAl12}. NSN, in particular, has been  interested in finding refined models for taking into account the
arbitrary and non-symmetrical distributions of traffic uncertainty characterizing
nation-wide optical networks \cite{Sc11}. Our modeling and theoretical developments
are thus also strongly based on practical industrial needs. A better modeling of
the traffic uncertainty affecting telecommunication networks was one
of the essential goals of our research activities in the German industrial research
project ROBUKOM \cite{BlDAHa11,KoEtAl12}. This project aimed at developing new models and algorithms for
the design of robust and survivable networks, in collaboration with NSN
 and the \emph{German National Research and Education Network} (DFN) \cite{DFN12}.
%
%
%
%

\bigskip

\noindent
\textbf{Outline of our contributions.}
In this work, we study the robust counterpart of a Mixed-Integer Linear Program (MILP) with uncertain coefficient matrix, when a \emph{multi-band uncertainty set} is considered. As previously noted, our main aim is to fill the knowledge gap on multi-band uncertainty in RO by presenting new theoretical results on this topic.
Specifically, we:
\begin{enumerate}
\item introduce a general definition of multi-band uncertainty (Section \ref{subsec:model});
\\
\item characterize a family of dominating uncertainty scenarios (Section \ref{subsec:model});
\\
\item derive a compact formulation for the robust counterpart of a MILP, whose size is smaller than that presented by us in \cite{BuDA12a,BuDA12b} (due to the new dominance results) (Section \ref{sec:compactMB});
\\
\item define an efficient method for the separation of robustness cuts (i.e., cuts that
    impose robustness), based on solving a min-cost flow instance. In particular, we propose an alternative proof to the one that we presented in \cite{BuDA12a}: the new proof takes advantage of the new dominance results and highlights the correspondence between the dominating uncertainty scenarios and the integral flows of an auxiliary min-cost flow instance (Section \ref{sec:rocuts});
\\
  \item study the properties of a special subfamily of uncertain MILPs, namely pure Binary Programs with uncertainty just affecting the objective function, and  we show that if the nominal problem is polynomially solvable, then also its counterpart remains polynomially solvable (Section \ref{sec:01_Opt});
  \\
  \item derive a \emph{data-driven} bound on the probability that an optimal robust solution is infeasible, due to deviations that are not considered by the uncertainty set. We define the bound \emph{data-driven}, as it exploits samples of the unknown arbitrary distribution of an uncertain coefficient (Section \ref{sec:prob_bound});
\end{enumerate}

\noindent
We remark that our original results are not obtained by simply extending the proofs of \cite{BeSi04} for single-band uncertainty, but need alternative proof strategies. For example, the separation of cuts imposing robustness does not anymore reduce to a trivial ordering problem (see \cite{FiMo12}), but, as we have proved, corresponds to solve a min-cost flow problem (see Section \ref{sec:rocuts} for details).

Moreover, our bound on the probability of robust infeasibility differs from those proposed in the two cornerstone RO papers \cite{BeNe00} and \cite{BeSi04}: in contrast to the bounds of these two works, that use a priori distributional information (symmetric distribution over symmetric single range),
we derive a \emph{data-driven bound} by exploiting samples of the unknown distribution. This type of bound is particularly suitable in real-world problems, where  historical data about the deviations are commonly available. Moreover, we think that a data-driven bound is also more in accordance with the spirit of multi-band uncertainty: multi-band is indeed particularly suitable to build strongly data-driven uncertainty sets, that can effectively approximate the shape of histograms built on historical data \cite{Bi07}.

Finally, we note that we first presented the results about the compact robust counterpart and the efficient separation of robustness cuts in \cite{BuDA12a}, where we focused attention on Linear Programs. A revised version of the work with the alternative proof for the separation of robustness cut is presented in \cite{BuDA12b}.
Additionally, we remark that in this paper we have decided to focus just on the theoretical results. However, we have already got promising computational results on realistic instances of real-world problems
(specifically, network design problems),
where the uncertainty set has been defined in collaboration with our industrial partners. We refer the reader to \cite{BuDA12a} and \cite{BuEtAl12} for a description of these results.


\bigskip

\noindent
\textbf{Review of related literature and comparisons with our work.}
While writing this paper, we became aware that a model similar to ours was presented in a very recent technical report (\cite{Ma12}). However, we remark that: 1) in our former papers \cite{BuDA12a}, \cite{BuDA12b} and \cite{BuDA12c}, we already introduced the multi-band model and formalized theoretical results similar to those
presented
in \cite{Ma12}; 2) the uncertainty model of \cite{Ma12} is a special case of our multi-band model, so our results are more general.
Moreover, in contrast to our model, the model of \cite{Ma12} cannot shape any arbitrary distribution, because of its limiting assumptions (symmetric deviation range, symmetric distribution of deviations, probability of deviation that decreases as the value of the deviation increases). Such hypotheses, in particular those on symmetry, greatly reduce the ability of modeling real uncertainty distributions and can result in excessive conservatism, as clearly pointed out even by Sim and colleagues in \cite{ChSiSu07}.


Reducing the conservatism of robust solutions associated with first classic works like \cite{BeNe00} and \cite{BeSi04} has been an important question approached in several works. We recall here some more significant references, referring the reader to \cite{BeElNe09,BeBrCa11} for a broader view. In \cite{ChSiSu07}, the conservatism is mitigated by defining new deviations measures, aimed at capturing asymmetrical distributions. In \cite{BeBeBr10}, a ``soft robust'' model is introduced by using the concept of risk measure, reducing the conservatism at the price of lower probabilistic guarantees of robustness.
The concept of risk measure is also used in \cite{BeBr09} to develop rules for constructing good uncertainty sets that reduce conservatism. Finally, another interesting approach based on the BS uncertainty model is constituted by ``light robustness'' \cite{FiMo09}: this approach adopts a hard upper bound on the objective value to control conservatism and aims at finding the most robust solution that satisfies the bound.

All these works successfully reach the objective of guaranteeing a reduction in conservatism, proposing new interesting approaches. However, at the same time, we think that they miss the elegant simplicity and easiness of access that characterize BS
and that have been undoubtedly decisive for its success also among practitioners and scholars from outside the Operations Research community. Concerning this point, we note that our work  differs from the previously cited ones: we defined and studied a general multi-band uncertainty model that overcomes remarkable limitations of BS,
while trying to maintain the simplicity of BS and offering straightforward ways to build uncertainty sets strongly based on available historical data. We remark that the simplicity of access and the data-driven characteristics have been highly requested by our industrial partners (e.g., Nokia Siemens Networks) in several industrial collaborations. Our modeling and theoretical developments are thus also inspired by and satisfy practical industrial needs \cite{KoEtAl12}.



\subsection{Model and Notation.} \label{subsec:model}

In this work, we study the robust counterpart of Mixed-Integer Linear Programming
Problems whose coefficient matrix is subject to uncertainty and the
uncertainty set is modeled through multiple deviation bands. We start by considering a
generic deterministic problem of the form:
\begin{eqnarray*}
\max &  & \sum_{j\in J}c_{j}\hspace{0.1cm}x_{j}\hspace{3.4cm}\mbox{(MILP)}\\
 &  & \sum_{j\in J}a_{ij}\hspace{0.1cm}x_{j}\leq b_{i}\hspace{1cm}i\in I\\
 &  & x_{j}\geq0\hspace{2.2cm}j\in J
 \\
 &  & x_{j} \in \mathbb{Z}_{+} \hspace{1.9cm} j \in J_{\mathbb{Z}} \subseteq J
\end{eqnarray*}

\noindent
where $I=\{1,\ldots,m\}$ and $J=\{1,\ldots,n\}$ denote the set
of constraint and variable indices, respectively. Assuming that the uncertainty affects only the coefficient matrix does not limit the generality of our study: indeed, 1) if the r.h.s. vector $b$ is uncertain, we introduce an additional variable $x_{n+1}: x_{n+1} = 1$ and include $b$ as a column of the matrix $A$; 2) if the cost vector $c$ is uncertain, we introduce an additional variable $L$, add the constraint $c'x \geq L$ and change the objective function into $\max L$.

We assume that the value of each coefficient $a_{ij}$ is uncertain and is equal to the summation of a \emph{nominal value} $\overline{a}_{ij}$ and a \emph{deviation} lying in the range $[d_{ij}^{K^{-}},d_{ij}^{K^{+}}]$, where $d_{ij}^{K^{-}},d_{ij}^{K^{+}} \in \mathbb{R}$ represent the maximum negative and positive deviations from $\overline{a}_{ij}$, respectively.
The \emph{actual value} $a_{ij}$ thus lies in the interval
$[\bar{a}_{ij}+d_{ij}^{K^{-}}, \hspace{0.1cm} \overline{a}_{ij}+d_{ij}^{K^{+}}]$.

We derive a generalization of the Bertsimas-Sim uncertainty model
by partitioning the single deviation band $[d_{ij}^{K-},d_{ij}^{K+}]$ of each coefficient $a_{ij}$ into $K$ bands, defined on the basis of $K$ deviation values:
$$
-\infty<
{d_{ij}^{K^{-}}<\cdots<d_{ij}^{-2}<d_{ij}^{-1}
\hspace{0.1cm}<\hspace{0.2cm}d_{ij}^{0}=0\hspace{0.2cm}<\hspace{0.1cm}
d_{ij}^{1}<d_{ij}^{2}<\cdots<d_{ij}^{K^{+}}}
<+\infty .
$$
Through these values, we define:
1) a set
of positive deviation bands, such that each band $k\in \{1,\ldots,K^{+}\}$
corresponds to the range $(d_{ij}^{k-1},d_{ij}^{k}]$;
2) a set of negative deviation bands, such that each band $k \in \{K^{-}+1,\ldots,-1,0\}$
corresponds to the range $(d_{ij}^{k-1},d_{ij}^{k}]$ and band $k = K^{-}$
corresponds to the single value $d_{ij}^{K^{-}}$
(the interval of each band but $k =K^{-}$ is thus open on the left).
With a slight abuse of notation, in what follows we indicate
a generic deviation band through the index $k$, with  $k \in K = \{K^{-},\ldots,-1,0,1,\ldots,K^{+}\}$
and the corresponding range by $(d_{ij}^{k-1}, d_{ij}^{k}]$.

Additionally, for each band $k\in K$, we define a lower bound $l_{k}$ and
an upper bound $u_{k}$ on the number of deviations that may fall in $k$, with $l_k, u_{k} \in \mathbb{Z}$ satisfying $0 \leq l_k \leq u_{k} \leq n$. In the case of band $k = 0$,
we assume that $u_{0}=n$, i.e. we
do not limit the number of coefficients that take
their nominal value. Furthermore, we assume that $\sum_{k \in K} l_k \leq n$, so that there always exists a feasible realization of the coefficient matrix.

We remark that, in order to avoid an overload of the notation, we assume that the number of bands $K$ and the bounds $l_k, u_k$ are the same for each constraint $i \in I$. Anyway, it is straightforward to modify all presented results to consider different values of those parameters for each constraint.

All the elements introduced above define a \emph{multi-band scenario set} $\SMB$. A \emph{scenario} $S$ specifies the deviation $d_{ij}^{S}$ of each coefficient $a_{ij}$ of the problem (thus,  $a_{ij} = \overline{a}_{ij} + d_{ij}^{S}$). We say that a scenario $S$ is \emph{feasible} and belongs to the set $\SMB$ if and only if the deviations $d^S_{ij}$ satisfy the following three properties:
\begin{eqnarray}
 \label{def:scenario_prop1}
  &&
  d^S_{ij}\in[d_{ij}^{K^{-}},d_{ij}^{K^{+}}]
  \hspace{4.6cm}
  i \in I, j \in J,
  \\
  \label{def:scenario_prop2}
  &&
  l_{k}
  \hspace{0.1cm} \leq \hspace{0.1cm}
  \left|
    \{j\in J: d^S_{ij} \in (d_{ij}^{k-1}, d_{ij}^{k}] \}
  \right|
  \hspace{0.1cm} \leq \hspace{0.1cm}
  u_{k}
  \hspace{1.0cm}
  i \in I, k \in K \backslash \{K^{-}\} ,
  \\
  \label{def:scenario_prop3}
  &&
  l_{K^{-}}
  \hspace{0.1cm} \leq \hspace{0.1cm}
  \left|
    \{j\in J: d_{ij}^{S}=d_{ij}^{K-}\}
  \right|
  \hspace{0.1cm} \leq \hspace{0.1cm}
  u_{K^{-}}
  \hspace{1.3cm}
  i \in I .
\end{eqnarray}

\noindent
Property (\ref{def:scenario_prop1}) enforces that the deviation of each coefficient lies in the corresponding overall deviation range. Properties (\ref{def:scenario_prop2})-(\ref{def:scenario_prop3}) guarantee that, for each constraint and deviation band, the number of deviations falling in the band must respect the corresponding bounds.

%
%
Before deriving a compact formulation for the robust counterpart of (MILP), we introduce
some structural properties of the multi-band scenario set $\SMB$.
To this end, we first define the \emph{band partition} of a scenario $S\in \mathcal{S}_M$ as the family of subsets:
\begin{eqnarray}
\label{band_partition}
J_{ik}(S)=\{j\in J: d^S_{ij}\in (d_{ij}^{k-1}, d_{ij}^{k}]\}, \qquad i\in I, k\in K \; .
\end{eqnarray}

\noindent
Each subset $J_{ik}(S)$ includes the indices $j \in J$ of constraint $i \in I$ whose deviation falls in band $k$.
It is easy to verify that such subsets define a partition of $J$ for each constraint $i$ (i.e., $\bigcup_{k\in K}J_{ik}(S)=J$,
$\forall \hspace{0.05cm} i\in I$, and $J_{ik}(S)\cap J_{ik'}(S)=\emptyset$, $\forall \hspace{0.05cm}
i\in I$, $k,k'\in K:$ $k\neq k'$). Furthermore, we call the \emph{profile
of a scenario} $S \in \mathcal{S}_M$  the numbers $p$ and $\theta_{k}$, $k \in K$ defined as follows:
\begin{eqnarray}
\label{profile_scenario_1}
p &=& \min\{k\in K: \sum_{i=K^{-}}^{k}l_{i}+\sum_{i=k+1}^{K^{+}}u_{i}\leq n\}
\\
&&
\nonumber
\\
\label{profile_scenario_2}
\theta_{k} &=&
\left\{
    \begin{array}{lll}
        l_{k}
        \hspace{2.9cm} k \leq p-1
        \\
        u_{k}
        \hspace{2.8cm} k \geq p+1
        \\
        n - \sum_{k\in K\backslash\{p\}}\theta_{k}
        \hspace{0.55cm} k=p \; .
    \end{array}
\right.
\end{eqnarray}

\medskip

\noindent
Note that $p\ge0$ since $u_{0}=n$ and $\sum_{k\in K}\theta_k =n$.

Finally, we define a \emph{dominance relation} among scenarios: let $S, S' \in \SMB$ be two feasible scenarios. Then $S$ \emph{dominates} $S'$ if $d^S_{ij} \geq d^{S'}_{ij}$
for all $i\in I$, $j\in J$. Obviously, a solution $x$ that is feasible under $S$
(i.e., $\sum_{j\in J} (\bar{a}_{ij} + d^S_{ij}) \hspace{0.1cm} x_{j} \leq b_{i}$, $\forall i\in I$)
remains feasible under $S'$. In the following lemma, we
prove that we can define the robust counterpart of (MILP) by limiting our attention to those
scenarios that satisfy the following properties: 1) each deviation attains the maximum value of the band in which it falls; 2) the cardinality of each subset of the band partition (\ref{band_partition}) equals the value associated with the profile (\ref{profile_scenario_2}) of the considered feasible scenario.

\begin{lemma}\label{lem:profile_satisfying}
Let $S\in \mathcal{S}_M$ be a feasible scenario.
\begin{enumerate}
\item If $d^S_{ij}\neq d_{ij}^{k}$ for any $j\in J_{ik}(S)$, $i\in I$,
$k\in K$, then there exists a feasible scenario $S'\in \mathcal{S}_M$ dominating $S$.
\item If $|J_{ik}(S)|\neq\theta_{k}$ for any $i\in I$, $k\in K$, then
there exists a feasible scenario $S'\in \mathcal{S}_M$ dominating $S$.
\end{enumerate}
\end{lemma}

\begin{proof}
1) Since $x \geq 0$ and the values falling in a band are bounded from
above by $d_{ij}^{k}$, a scenario $S'$ such that $d^{S'}_{ij}=d_{ij}^{k}$
for each $j\in J_{ik}(S)$, $i\in I$, $k\in K$ dominates the scenario~$S$.

2) Let $k^1$ be the highest band $k \in K$ such that $|J_{ik}(S)|<\theta_{k}$. Then
there exists a band $k^2<k^1$ such that $|J_{ik^2}(S)|>\theta_{k^2}$. Let
$\mu \in J_{ik^2}$. We define the feasible scenario $S'$ by setting
$d^{S'}_{i\mu}=d_{i\mu}^{k^1}$ and $d^{S'}_{ij}=d^{S}_{ij}$ for $j \neq \mu$. Then $S'$
dominates $S$. \qed
\end{proof}

\noindent
We call a scenario $S\in \SMB$ that satisfies the properties of Lemma~\ref{lem:profile_satisfying} \emph{profile and bound valid}. Due to Lemma~\ref{lem:profile_satisfying}, in the rest of the paper we can limit our attention to these scenarios.
We now proceed to study the robust counterpart of (MILP) under multi-band uncertainty.

%

%
%

\section{A Compact Robust MILP Counterpart} \label{sec:compactMB}

The robust counterpart of the problem (MILP) under a multi-band scenario set $\SMB$
can be equivalently written as:
\begin{eqnarray*}
\max && \sum_{j\in J}c_{j}\hspace{0.1cm}x_{j}
\\
 && \sum_{j\in J} \overline{a}_{ij} \hspace{0.1cm} x_{j}
 + DEV_i(x,\SMB) \leq b_{i}
 \hspace{1.4cm} i\in I
 \\
 &  & x_{j} \geq 0 \hspace{5.05cm}j\in J
 \\
 &  & x_{j} \in \mathbb{Z}_{+} \hspace{4.8cm} j \in J_{\mathbb{Z}} \subseteq J ,
\end{eqnarray*}

\noindent
where DEV$_i(x,\SMB)$ is the maximum total deviation induced by the scenario set $\SMB$  for a feasible solution $x$ when constraint $i$ is considered.
Due to Lemma~\ref{lem:profile_satisfying}, DEV$_i(x,\SMB)$ is equal to the optimal value of the following pure 0-1 Linear Program (note that in this case index $i$ is fixed):
\begin{eqnarray} \label{DEV01}
DEV_i(x,\SMB) = &\max&
\sum_{j \in J} \sum_{k \in K} d^{k}_{ij} \hspace{0.1cm} x_{j} \hspace{0.1cm} y_{ij}^{k}
\hspace{1.5cm} \mbox{(DEV01)}
\nonumber
\\
&&
\sum_{j \in J} y_{ij}^{k}
\hspace{0.1cm} = \hspace{0.1cm}
\theta_k
\hspace{2.2cm}
k \in K
\label{DEV01_bounds}
\\
&&
\sum_{k \in K} y_{ij}^{k} \leq 1
\hspace{2.5cm}
j \in J
\label{DEV01_singleBand}
\\
&&
y_{ij}^{k} \in \{0,1\}
\hspace{2.5cm}
j \in J, k \in K.
\label{DEV01_bandVars}
\end{eqnarray}

\noindent
The generic binary variable (\ref{DEV01_bandVars}) is equal to 1 if coefficient
${a}_{ij}$ deviates in band $k$ and is 0 otherwise.
The constraints (\ref{DEV01_bounds}) impose the profile of  $\SMB$ on the number of coefficients deviating in each band $k$. The constraints (\ref{DEV01_singleBand}) ensure that each coefficient deviates in at most one band (actually these should be equality constraints, but, for the assumption $u_{0}=n$ made in Section \ref{subsec:model}, we can consider inequalities). An optimal solution of (DEV01) thus defines a distribution of the coefficients among the bands which maximizes the total deviation w.r.t.~the nominal values, while respecting the bounds on the number of deviations of each band.

\hspace{0.5cm}

\noindent
We now prove that the robust counterpart of (MILP) under a multi-band scenario set $\SMB$ can be reformulated as a \emph{compact} Mixed-Integer Linear Program. To this end, we first consider first the \emph{linear relaxation} (DEV01-RELAX) of (DEV01) and show that it presents the following nice property.
\begin{proposition} \label{th:integrality_DEV01}
The polytope described by the constraints of (DEV01-RELAX) is integral.
\end{proposition}

\begin{proof}
We prove the integrality by showing that the coefficient matrix is totally unimodular.
Consider first the linear relaxation of (DEV01):
\begin{eqnarray}\label{DEV01-relax}
&\max& \sum_{j\in J}\sum_{k\in K}d_{ij}^{k} \hspace{0.1cm} x_{j} \hspace{0.1cm} y_{ij}^{k}
\hspace{2.0cm} \mbox{(DEV01-RELAX)}
\nonumber
\\
&&
\sum_{j\in J} y_{ij}^{k} = \theta_{k}
\hspace{2.9cm}
k\in K
\label{DEV01-relax-cnstr1}
\\
&&
\sum_{k\in K}y_{ij}^{k} \leq 1
\hspace{2.95cm} j\in J
\label{DEV01-relax-cnstr2}
\\
&&
y_{ij}^{k} \geq 0
\hspace{3.6cm} j\in J, k\in K ,
\label{DEV01-relax-cnstr3}
\end{eqnarray}
where we dropped constraints $y_{ij}^{k} \leq 1$, that are dominated by constraints (\ref{DEV01-relax-cnstr2}).

\bigskip

\noindent
The constraints of (DEV01-RELAX) have  the following matrix form:
\[
F_i
\hspace{0.1cm}
y_i
\hspace{0.1cm}
=
\hspace{0.1cm}
{\small
    \left(\begin{array}{c|c|c|c}
    &  & \\
    I & I & \mbox{ } \cdots \mbox{ } & I\\
    &  &\\
    \hline
    1 \cdots 1 & & \\
     & 1\cdots1 & &\\
     & & \mbox{ } \ddots \mbox{ }  &\\
     & &  & 1\cdots1 \\
    \end{array}\right)
\hspace{0.2cm}
    \left(\begin{array}{c}
    y_{i1}^{K^{-}}\\
    \vdots\\
    y_{i1}^{K^{+}}\\
    \hline
    \vdots\\
    y_{ij}^{k}\\
    \vdots\\
    \hline
    y_{in}^{K^{-}}\\
    \vdots\\
    y_{in}^{K^{+}}\\
    \end{array}\right)
\hspace{0.1cm}
\begin{array}{c}
= \\
\le
\end{array}
\hspace{0.1cm}
    \left(\begin{array}{c}
    \vdots\\
    \theta_k\\
    \vdots\\
    \hline
    \vdots\\
    1\\
    \vdots\\
    \end{array}\right)
    }
\hspace{0.1cm}
=
\hspace{0.1cm}
g_i
\]

\noindent
%
%
It is easy to verify that $F_i$ is the incidence matrix of a bipartite graph: the elements of the two disjoint set of nodes of the graph are in correspondence with the rows of the two distinct layers of block in $F_i$. Moreover, every column has exactly two elements that are not equal to zero: one in the upper layer of blocks and one in the lower layer.  Being the incidence matrix of a bipartite graph, $F_i$ is a totally unimodular matrix \cite{NeWo88}.
Since $F_i$ is totally unimodular and the vector $g_i$ is integral, it is well-known that the polytope defined by $F_i \leq g_i$ and $y_i \geq 0$ is integral, thus completing the proof.
\qed
\end{proof}

\noindent
Due to this integrality property, we can use strong duality to prove the main result of this section.
\begin{theorem} \label{th:compact_RLP}
The robust counterpart of a Mixed-Integer Linear Program under the multi-band scenario set $\SMB$ is equivalent to the following compact Mixed-Integer Linear Program:
\begin{eqnarray*}\label{robustCompactLP}
\max && \sum_{j\in J} c_{j} \hspace{0.1cm} x_{j}
\hspace{5.6cm} \mbox{(Rob-MILP)}
\\
&& \sum_{j\in J} \bar{a}_{ij} \hspace{0.1cm} x_{j}
+ \sum_{k\in K} \theta_{k} \hspace{0.1cm} w_{i}^{k}
+ \sum_{j\in J} z_{ij} \leq b_{i}
\hspace{1.4cm} i\in I
\\
&&  w_{i}^{k} + z_{ij} \geq d_{ij}^{k} \hspace{0.1cm} x_{j}
\hspace{4.30cm} i \in I, j\in J, k\in K
\\
&& z_{ij} \geq 0
\hspace{5.85cm} i \in I,j\in J
\\
&& w_{i}^k \in \mathbb{R}
\hspace{5.8cm} i\in I, j\in K
\\
&& x_{j} \geq 0
\hspace{5.9cm} j\in J
\\
&& x_{j} \in \mathbb{Z}_{+}
\hspace{5.6cm} j \in J_{\mathbb{Z}} \subseteq J \;.
\end{eqnarray*}
\end{theorem}

\begin{proof}
As first step, consider the dual problem of (DEV01-RELAX):
\begin{eqnarray*}\label{dual_DEV_01}
\min &&
\sum_{k\in K} \theta_{k} \hspace{0.1cm} w_{i}^{k} + \sum_{j\in J} z_{ij}
\hspace{2.8cm}  \mbox{(DEV01-RELAX-DUAL)}
\\
&& w_{i}^{k} + z_{ij} \geq d_{ij}^{k} \hspace{0.1cm} x_{j}
\hspace{3.15cm} j\in J, k\in K
\\
&& z_{ij} \geq 0
\hspace{4.7cm} j\in J
\\
&&  w_{i}^{k} \in \mathbb{R}
\hspace{4.6cm} k \in K \; ,
\end{eqnarray*}

\noindent
where the dual variables $w_{i}^{k}, z_{ij}$ are respectively associated with the primal constraints (\ref{DEV01-relax-cnstr1}, \ref{DEV01-relax-cnstr2}) of (DEV01-RELAX) defined for constraint $i$.

Since (DEV01-RELAX) is feasible and bounded by definition of the scenario set $\SMB$, also (DEV01-RELAX-DUAL) is feasible and bounded and the optimal values of the two problems are the same (strong duality). Then, by Proposition \ref{th:integrality_DEV01}, we can replace the maximum total deviation $DEV_i(x,\SMB)$ with problem (DEV01-RELAX-DUAL), obtaining the compact Mixed-integer Linear Program (Rob-MILP). This concludes the proof.
\qed
\end{proof}

\noindent
In comparison to (MILP), this compact formulation uses $K \cdot m + n \cdot m$ additional
variables and includes $K \cdot n \cdot m$ additional constraints.
Due to the dominance results, we remark that the new formulation (Rob-MILP) has $K \cdot m$ variables less than the compact counterpart that we presented in \cite{BuDA12a}.
Moreover, we note that when (MILP) does not include integer variables (i.e., $J_{\mathbb{Z}} = \emptyset$) and is thus a pure Linear Program, (Rob-MILP) is a pure Linear Program as well.

%
%

%
%

\section{Separation of Robustness Cuts}\label{sec:rocuts}

As an alternative to directly solving the compact robust counterpart (Rob-MILP), we have investigated the development of a cutting-plane algorithm: we start by solving the nominal problem (MILP) and then test if the optimal solution is robust. If not, we generate a cut that imposes robustness (\emph{robustness cut}) and we add it to the problem. This initial step can then be iterated as in a typical cutting plane method \cite{NeWo88}.

Testing whether a solution $x \in \mathbb{R}_+^{n - |J_{\mathbb{Z}}|} \times \mathbb{Z}_+^{|J_{\mathbb{Z}}|}$ is robust, corresponds to confirming that
$\sum_{j \in J} \bar{a}_{ij} \hspace{0.1cm} x_{j} + DEV_i(x,\SMB) \leq b_i$ for every constraint $i\in I$. In the case of the Bertsimas-Sim model,  this is very simple and merely requires to sort the deviations and choose the largest $\Gamma_i$ values \cite{FiMo12}. In the case of a multi-band scenario set, this simple approach does not guarantee the robustness of a computed solution.
However, in the following theorem we prove that testing the robustness of a solution corresponds to solving a \emph{min-cost flow problem} \cite{AhMaOr93}.
We recall that an integral min-cost flow of value $n$ can be computed in polynomial time \cite{AhMaOr93}.

\begin{theorem} \label{theorem:robCut}
Let $x \in \mathbb{R}_+^{n - |J_{\mathbb{Z}}|} \times \mathbb{Z}_+^{|J_{\mathbb{Z}}|}$ and let $\SMB$ be a multi-band scenario set.
Moreover, let $(G,c)_{x}^{i}$ be the min-cost flow instance corresponding to $x$ and a constraint $i \in I$ of (MILP) and built according to the rules specified in the proof.

\noindent
The solution $x$ is robust for constraint $i$ w.r.t. $\SMB$ if and only if
$$
\bar{a}_{i}' x - c^{*}_i(x) \leq b_{i}
$$
where $c^{*}_i(x)$ is the value of  a min cost flow in the instance $(G,c)_{x}^{i}$.
\end{theorem}

\begin{proof}
We prove the statement by showing that there exists a one-to-one correspondence between the scenarios $S\in \SMB$ satisfying the profile and bound properties (see Lemma~\ref{lem:profile_satisfying}) and  the integral flows $f\in F_i$ associated with the min-cost flow instance $(G,c)_x^i$ (defined below), with cost relation
\[ c(f) = -\sum_{j\in J} d_{ij}^S x_j \;. \]
Thus,
\[ DEV_i(x,\SMB)  = \max_{S\in \SMB} \sum_{j\in J} d_{ij}^S x_j
= -\min_{f\in F_i} c(f) = - c_i^*(x),\]
which concludes the proof.

The min-cost flow instance $(G,c)_{x}^{i}$ is defined as follows.
The directed graph $G$  contains one vertex $v_j$ for each variable index $j \in J$, one vertex $w_k$ for each band $k \in K$ and two vertices $s,t$ that are the source and the sink of the flow, i.e.~$V=\bigcup_{j\in J} \{v_{j}\} \cup \bigcup_{k\in K}\{w_{k}\} \cup \{s,t\}$.
The set of arcs $A$ is the union of three sets $A_1, A_2, A_3$. The set $A_1$ contains one arc from $s$ to every vertex $v_{j}$, i.e. $A_{1}=\{(s,v_{j}): j\in J\}$.
The set $A_2$ contains one arc from every  vertex $v_{j}$ to every  vertex $w_{k}$, i.e. $A_{2}=\{(v_{j},w_{k}): j\in J, k\in K\}$. Finally, the set $A_3$ contains one arc from every  vertex $w_{k}$ to the sink $t$, i.e. $A_{3}=\{(w_{k},t):  k\in K\}$.
By construction, $G(V,A)$ is bipartite and acyclic.
To each arc $a \in A$, we associate an upper bound $u_a$ on the flow that can be sent on $a$ and  a cost $c_a$ of sending one unit of flow on $a$.
These two values are set in the following way: $(1,0)$ when $a=(s,v_j)\in A_{1}$; $(1,-d_{ij}^{k}x_{j})$ when $a=(v_j,w_k)\in A_{2}$; $(\theta_k,0)$ when $a=(w_k,t)\in A_{3}$.
Finally, the amount of flow that must be sent trough the network from $s$ to $t$ is equal to $n$. We denote by $F_i$ the set of feasible integral flows of value $n$ of the min-cost flow instance $(G,c)_{x}^{i}$.
The \emph{cost} of an $(s,t)$-flow $f \in F_i$ is defined by $c(f)=\sum_{a\in A} c_{a}f_{a}$.

Given a scenario $S\in \SMB$ that satisfies the profile and bound properties, we construct an integral flow $f$ for the min-cost flow instance $(G,c)_{x}^{i}$, in the following way:
\begin{eqnarray}
&&
  f_{(s, v_j)} = 1
        \hspace{3.7cm} \forall j \in J
  \label{flowDef1}
  \\
&&
  f_{(v_j,w_k)} = 1
        \hspace{0.1cm} \Longleftrightarrow \hspace{0.1cm}
        d_{ij}^S = d_{ij}^{k}
        \hspace{1.1cm} \forall j \in J, k \in K
  \label{flowDef2}
  \\
&&
  f_{(w_k, t)} = \theta_k
        \hspace{3.5cm} \forall k \in K .
  \label{flowDef3}
\end{eqnarray}

\noindent
The flow $f$ respects the arc capacities and has a value of $n$. Since each deviation is assigned by $S$ to exactly one band, $f$ also satisfies flow conservation in each vertex $v_j$. Furthermore, $S$ assigns exactly $\theta_k$ values to band $k$, and thus
\[
\sum_{j\in J} f_{(v_j,w_k)} = |J_{ik}(S)|=\theta_k = f_{(w_k,t)} \;.
\]
Hence, $f$ is in $F_i$ and by definition of the arc costs $-c(f)= \sum_{j\in J} d_{ij}^S x_j$.

Given a feasible integral flow $f\in F_i$ of value $n$ and cost $c(f)$ of the min-cost flow instance $(G,c)_x^i$, we construct a scenario $S\in \SMB$
satisfying the profile and bound properties in the following way:
\[
d_{ij}^S = d_{ij}^k \Longleftrightarrow f_{(v_j,w_k)}=1 \hspace{1.0cm} \forall j\in J, k\in K.
\]
This definition is valid, since each arc $(s,v_j)$ carries one unit of flow and thus due to flow conservation exactly one arc $(v_j, w_k)$ carries flow. Furthermore, since $\sum_{k\in K} \theta_k =n$, exactly $\theta_k$ arcs with positive flow value enter $w_k$ and therefore $|J_{ik}(S)|= \theta_k$. Summing up,
\[
c(f)=\sum_{j\in J}\sum_{k\in K} - d_{ij}^k x_j f_{(v_j,w_k)}
= \sum_{k\in K} \sum_{j\in J_{ik}(S)} -d_{ij}^S x_j
= \sum_{j\in J} -d_{ij}^S x_j. \]
This concludes the proof.
\qed
\end{proof}

\noindent
From Theorem \ref{theorem:robCut}, we can immediately derive the following corollary:

\begin{corollary} \label{corollary:robCut}
A solution  $x \in \mathbb{R}_+^{n - |J_{\mathbb{Z}}|} \times \mathbb{Z}_+^{|J_{\mathbb{Z}}|}$ is robust w.r.t. $\SMB$ if and only if
$$
\bar{a}_{i}' x - c^{*}_i(x) \leq b_{i}
\hspace{1cm}
\forall i \in I,
$$
where $c^{*}_i(x)$ is the value of a minimum cost flow in $(G,c)_{x}^{i}$.
\end{corollary}

\noindent
According to this corollary, we can test the robustness of a solution  $x \in \mathbb{R}_+^{n - |J_{\mathbb{Z}}|} \times \mathbb{Z}_+^{|J_{\mathbb{Z}}|}$
by computing an integral flow $f^{i}$ of minimum cost $c^{*}_i(x)$ in $(G,c)_{x}^{i}$ for every $i \in I$.
If $\bar{a}' x - c^{*}_i(x) \leq b_{i}$ for every $i\in I$,
then $x$ is robust feasible. Otherwise $x$ is not robust and there exists an index
 $i \in I$ such that $\bar{a}' x - c^{*}_i(x) > b_{i}$ and thus
\begin{equation}\label{eq:robCut}
\sum_{j\in J} \bar{a}_{ij} \hspace{0.1cm} x_{j}
+ \sum_{j\in J} \sum_{k\in K} d_{ij}^{k} \hspace{0.1cm} x_{j} \hspace{0.1cm} f^i_{(v_j,w_k)} \leq b_{i}
\end{equation}
is valid for the polytope of the robust
solutions and cuts off the solution $x$.

%
%


\section{Binary Programs with Cost Uncertainties \label{sec:01_Opt}}

In this section, we present results on a subfamily of Mixed-Integer Linear Programs, namely pure Binary Programs, and consider the case in which uncertainty only affects the objective function. Specifically, we investigate Binary Programs of the form:
\begin{eqnarray*}
\min && \hspace{0.1cm} \sum_{j\in J}c_{j} \hspace{0.05cm} x_{j} \hspace{3.4cm} (BP)
\\
&&
x \in X \subseteq \{0,1\}^n
\end{eqnarray*}

\noindent
with non-negative cost vector, i.e. $c_j \geq 0$, for all $ j \in J = \{1,\ldots,n\}$. Relevant problems like the minimum spanning
tree problem, the maximum weighted matching problem and the shortest
path problem belong to this class of problems.

We study the robust version of (BP) under multi-band uncertainty affecting just the cost vector.  More formally, for each element
$j\in J$ we are given the nominal cost $\bar{c}_{j}$ and a sequence of
$K^{+} + 1$ deviation values $d_{j}^{k}$, with  $k\in K=\{0,\ldots,K^{+}\}$,
such that $0 = d_{j}^{0} <d_{j}^{1}<\ldots<d_{j}^{K+}<\infty$ (note that in contrast to
Section \ref{sec:compactMB} we consider only positive deviations). Through these values, we define: 1) the zero-deviation band
corresponding to the single value $d_{j}^{0}=0$; 2) a set $K^{+}$
of positive deviation bands, such that each band $k\in K \backslash \{0\}$
corresponds to the range $(d_{j}^{k-1},d_{j}^{k}]$.
Furthermore,  integer values $l_k,u_k \in \mathbb{Z}, 0 \leq l_k \leq u_k \leq n$
 represent the lower and upper bounds on the number
of deviations falling in each band $k \in K$.

Since (BP) is a special case of (MILP), we can define its robust counterpart (Rob-BP) by applying Lemma~\ref{lem:profile_satisfying} and Theorem~\ref{th:compact_RLP}:
\begin{eqnarray}
\min
&&
\sum_{j\in J}c_{j} \hspace{0.05cm} x_{j}
+ \sum_{k\in K} \theta_{k} \hspace{0.05cm} w_{k}+\sum_{j\in J} z_{j}
\hspace{2cm}  (Rob\mbox{-}BP)
\nonumber
\\
&&
w_{k}+z_{j}\ge d_{j}^{k} \hspace{0.05cm} x_{j}
\hspace{2.9cm}
j\in J, k\in K
\label{cnstr_vzxROB01}
\\
&&
w_{k} \geq 0
\hspace{4.2cm}
k\in K
\nonumber
\\
&&
z_{j} \geq 0
\hspace{4.3cm}
j\in J
\nonumber
\\
&&
x \in X ,
\nonumber
\end{eqnarray}

\noindent
In order to get a robust optimal solution, we can solve (Rob-BP) by means of a commercial MIP solver. However, as an alternative, in what follows we prove that we can get a robust optimal solution by solving a sequence of nominal problems (BP) with modified objective function.

\medskip

\begin{remark} \label{reamrk:noBands}
Without loss of generality, we can assume that the number of bands $K$ is constant and such
that $K \leq n$. Indeed, if $K > n$ then $K' = K - n$ bands have profile $\theta_k = 0$ and no coefficient is assigned to them, so they can be eliminated from the problem.
\end{remark}

\bigskip
\noindent
As first step, in the following lemma we prove that for a fixed vector $\bar{w} \geq 0^{K}$, solving (Rob-BP) corresponds to solving (BP) with modified cost coefficients.
\begin{lemma} \label{lemma:fixedW}
Given $\bar{w} \geq 0^{K}$, solving (Rob-BP) with $w = \bar{w}$ is equivalent to solving the following problem:
\begin{eqnarray*}
\min && \hspace{0.1cm} \sum_{j\in J} \bar{c}_{j} \hspace{0.05cm} x_{j}
\hspace{2.0cm}
( \hspace{0.05cm} Rob\mbox{-}BP(\bar{w}) \hspace{0.05cm})
\\
&& x \in X \; ,
\end{eqnarray*}
with
\begin{eqnarray*}
&&
\bar{c}_j = c_j + \bar{d}_j \hspace{0.5cm} \forall j \in J \; ,
\\
&&
\bar{d}_j = \max \hspace{0.1cm} \{ 0, \hspace{0.1cm} \max_{k \in K} \hspace{0.1cm} \{d_j^{k} - \bar{w}_k\} \} \hspace{0.5cm} \forall j \in J  \; .
\end{eqnarray*}
%
\end{lemma}

\begin{proof}
If $w = \bar{w}$ then (Rob-BP) reduces to:
\begin{eqnarray}
\sum_{k\in K} \theta_{k} \hspace{0.05cm} \bar{w}_{k}
+ \hspace{0.1cm}
\min
&&
\sum_{j\in J} c_{j} \hspace{0.05cm} x_{j}
+ \sum_{j\in J} z_{j}
\nonumber
\\
&&
z_{j}\ge d_{j}^{k} \hspace{0.05cm} x_{j} - \bar{w}_{k}
\hspace{1.4cm}
j\in J, k\in K
\label{cnstr_robFixedW_1}
\\
&&
z_{j} \geq 0
\hspace{2.8cm}
j\in J
\label{cnstr_robFixedW_2}
\\
&&
x \in X ,
\nonumber
\end{eqnarray}

\noindent
Because of the binary nature of $x_j$, constraints (\ref{cnstr_robFixedW_1}-\ref{cnstr_robFixedW_2})
can be substituted by $z_{j} \geq \bar{d}_j \hspace{0.05cm}  x_{j}$ since:
%
%
\begin{eqnarray*}
\left.
    \begin{array}{lll}
         z_{j} \geq d_{j}^{k} \hspace{0.05cm} x_{j} - \bar{w}_{k},
            \hspace{0.1cm}
            \forall k \in K
         \\
         z_{j} \geq 0
    \end{array}
\right.
\hspace{0.25cm} \Longleftrightarrow \hspace{0.25cm}
z_{j} &\geq& \max \hspace{0.05cm} \{ 0, \hspace{0.05cm}
\max_{k \in K} \hspace{0.05cm} \{d_j^{k} - \bar{w}_k\} \} \hspace{0.1cm} x_{j}
\hspace{0.1cm} = \hspace{0.1cm}
\bar{d}_j \hspace{0.05cm} x_{j}
\end{eqnarray*}

\noindent
We note that $\bar{d}_j \hspace{0.05cm} x_{j} \geq 0$ by the definition of $\bar{d}_j$.
Moreover, we note that an optimal solution $(\bar{x},\bar{z})$ is such that
$\bar{z}_{j} = \bar{d}_j  \bar{x}_{j}$, $\forall j \in J$. Indeed, suppose by contradiction that w.l.o.g. there exists a unique $\ell \in J: \bar{z}_{\ell} > \bar{d}_{\ell} \hspace{0.1cm} \bar{x}_{\ell}$. Then we can define a feasible solution $(x',z')$ by setting:
\begin{eqnarray*}
x'_j &=& \bar{x}_j \hspace{1.5cm} \forall j \in J
\\
z'_j &=& \bar{z}_j \hspace{1.55cm} \forall j \in J \backslash \{\ell\}
\\
z'_\ell &=& \bar{z}_\ell - s
\end{eqnarray*}


\noindent
with $s = \bar{z}_{\ell} - \bar{d}_{\ell} \hspace{0.05cm} \bar{x}_{\ell} > 0$,
thus obtaining a reduction $s$ in cost and contradicting the optimality of $(\bar{x},\bar{w})$.
As $z_{j} = \bar{d}_{j} \hspace{0.05cm} x_{j}$ at the optimum, for all $ j \in J$, we obtain:
\begin{eqnarray}
\sum_{k\in K} \theta_{k} \hspace{0.05cm} \bar{w}_{k}
+ \hspace{0.1cm}
\min
&&
\sum_{j\in J} c_{j} \hspace{0.05cm} x_{j}
+ \sum_{j\in J} \bar{d}_j \hspace{0.05cm} x_{j}
\nonumber
\\
&&
x \in X ,
\nonumber
\end{eqnarray}

\noindent
thus completing the proof.
\qed
\end{proof}

\medskip

\noindent
From this lemma it follows immediately that if $w^*$ belongs to some optimal solution of (Rob-BP), we can solve (Rob-BP($w^*$)) and obtain an optimal robust solution $x^*$.
\begin{corollary} \label{corollary:fixedW}
If $(x^*, z^*, w^*)$ is an optimal solution to (Rob-BP) and $Z^*$ is the corresponding optimal value, then $Z^*$ is such that:
$$
Z^* = \sum_{k\in K} \theta_{k} \hspace{0.05cm} w_{k}^*
    \hspace{0.05cm} + \hspace{0.05cm}
    \min_{x \in X} \sum_{j\in J} \bar{c}_{j} \hspace{0.05cm} x_{j} \; .
$$
\end{corollary}
\begin{proof}
Let ${\cal{F}}$ and ${\cal{F}}(\bar{w})$ be the sets of feasible solutions of (Rob-BP) and (Rob-BP($\bar{w}$)) for some fixed $\bar{w}\ge 0^K$, respectively. Then
\begin{eqnarray*}
  Z^* &=&
  \min_{(x,w,z) \in {\cal{F}}}
  \hspace{0.1cm}
    \sum_{j\in J}c_{j} \hspace{0.05cm} x_{j}
    + \sum_{k\in K} \theta_{k} \hspace{0.05cm} w_{k}+\sum_{j\in J} z_{j}
  \\
  &=&
  \sum_{j\in J}c_{j} \hspace{0.05cm} x_{j}^*
    + \sum_{k\in K} \theta_{k} \hspace{0.05cm} w_{k}^* +\sum_{j\in J} z_{j}^*
  \\
  &=&
    \sum_{k\in K} \theta_{k} \hspace{0.05cm} w_{k}^*
    +   \min_{(x,z) \in {\cal{F}}(w^*)} \hspace{0.1cm} \sum_{j\in J}c_{j} \hspace{0.05cm} x_{j} +\sum_{j\in J} z_{j}
  \\
  &\stackrel{(a)}{=}&
    \sum_{k\in K} \theta_{k} \hspace{0.05cm} w_{k}^*
    +   \min_{x \in X} \hspace{0.1cm} \sum_{j\in J} c_{j} \hspace{0.05cm} x_{j} +\sum_{j\in J} \bar{d}_{j} \hspace{0.05cm} x_{j}
  \\
  &=&
    \sum_{k\in K} \theta_{k} \hspace{0.05cm} w_{k}^*
    +
    \min_{x \in X} \sum_{j\in J} \bar{c}_{j} \hspace{0.05cm} x_{j}
\end{eqnarray*}
where the equality (a) holds since $z_{j} = \bar{d}_{j} x_{j}$ at the optimum, as showed in the proof of Lemma \ref{lemma:fixedW}.
\qed
\end{proof}

\noindent
Using the previous result,
a naive approach to solve (Rob-BP) would be to test the value of (Rob-BP($w$)) for every $w \geq 0^{K}$. However, this would require an exponential number of tests.
We prove instead that we can limit our attention to a subset of vectors $w$ of polynomial size, as indicated in the algorithm and theorem that follows.

\begin{algo}[H]
\label{algorithmRobBP}
\begin{lyxlist}{00.00.0000}
\item [{Input:}]
set of feasible solutions $X\subseteq\{0,1\}^{n}$;
\\
cost values $c_{j}\ge0$ for all $j\in J$;
\\
deviation values $0=d_{j}^{0}\le d_{j}^{1}<d_{j}^{2}<\ldots<d_{j}^{K^{+}}$
for all $j\in J$;
\\
lower and upper bounds $l_{k},u_{k}$ for all $k\in K$.
\\
\item [{Output:}] an optimal robust solution $x^{*} \in X$
\\
\item [{Step~1:}]
a) set $Z = \infty$, $(\bar{x},\bar{w},\bar{z}) = $ NULL
\\
b) define $B_{kk'} = \{d_{j}^{k} - d_{j}^{k'}, \forall j \in J\} \cup \{0\}$,
    $\forall k, k' \in K: k \neq k'$
\\
c) define $C_{k} = \{d_{j}^{k}, \forall j \in J\} \cup \{0\}$,
    $\forall k \in K$
\\
\item [{Step~2:}] FOR each possible combination of values $b_{kk'} \in B_{kk'}$, $c_k \in C_{k}$ compute a solution $w$ to the linear system
        \begin{eqnarray*}
        w_{k} - w_{k'} &\geq& b_{kk'}
        \hspace{1.00cm}
        k, k' \in K: k \neq k'
        \\
        w_{k} &\leq& c_{k}
        \hspace{1.25cm}
        k \in K
        \\
        w_{k} &\geq& 0
        \hspace{1.40cm}
        k \in K
        \end{eqnarray*}
IF a solution $w$ exists THEN
    \begin{description}
      \item a) solve (Rob-BP($w$)) and let $(x(w),w,z(w))$, $Z(w)$ be the corresponding optimal solution and value, respectively
      \item b) IF $\theta' w + Z(w) < Z$ THEN
            \begin{description}
              \item i) set $Z = \theta' w + Z(w)$ and $(\bar{x},\bar{w},\bar{z}) = (x(w),w,z(w))$
              \\
            \end{description}
    \end{description}
\item [{Step~3:}] return $(\bar{x},\bar{w},\bar{z})$
\end{lyxlist}
\caption{\label{alg:robust_solution}Computing an optimal robust solution of (Rob-BP)}
\end{algo}

\begin{theorem} \label{theorem:algComplexity}
Algorithm 1.1 computes a robust optimal solution to (Rob-BP) by solving at most $(n + 1)^{k^2}$ nominal problems
with modified cost coefficients
(Rob-BP($w$)).
If there exists an algorithm to solve the nominal problem
(BP) in polynomial time, Algorithm 1.1 has a polynomial run-time for a
constant number $K$ of deviation bands.
\end{theorem}
\begin{proof}
Let $x^*$ be an optimal robust solution and let $S$ be a scenario in $\mathcal{S}_M$ such that the deviation on $x^*$ is maximized. We then define a partition of $J$ according to the bands in which every deviation value $d^S(j)$ falls, i.e. $J_k=\{j \in J : d^S_j \in (d^{k-1}_j, d^k_j]\}$
(this is a partition as every coefficient deviates in exactly one band).

\bigskip

\noindent
\emph{Claim 1.} Let $x^{*}$ be a robust optimal solution. Then there exists $(w^{*}, z^{*})$ such that $(x^{*}, w^{*}, z^{*})$ is an optimal solution of (Rob-BP) and $w^{*}$ is a feasible solution of the following system $LP(b_{kk'},c_{k})$:
\begin{eqnarray*}
w_{k}^{*} - w_{k'}^{*} &\geq& b_{kk'}
\hspace{1.00cm}
k, k' \in K: k \neq k'
\hspace{0.75cm}  (LP(b_{kk'},c_{k}))
\\
w^{*}_{k} &\leq& c_{k}
\hspace{1.25cm}
k \in K
\\
w^{*}_{k} &\geq& 0
\hspace{1.40cm}
k \in K
\end{eqnarray*}

\noindent
with $c_k = \min_{j \in J_k} \{d_{j}^{k} \hspace{0.05cm} x_{j}^{*}\}$, for all $ k \in K$
and $b_{kk'} = \max_{j \in J_{k'}} \{d_{j}^{k} \hspace{0.05cm} x_{j}^{*} - d_{j}^{k'} \hspace{0.05cm} x_{j}^{*} \}$, for all $k, k' \in K: k \neq k'$. Moreover, $z^{*}_{j} = \hspace{0.1cm} d_{j}^{k} \hspace{0.05cm} x_{j}^{*} - w_{k}^{*}$,  for all $k \in K, j \in J_k$.

\bigskip
\noindent
To prove the claim, we first note that since $x^{*}$ is a robust optimal solution, the corresponding values $w^{*}, z^{*}$ minimize the following linear program obtained from (Rob-BP) for $x = x^{*}$:
\begin{eqnarray}
\min
&&
\sum_{k\in K} \theta_{k} \hspace{0.05cm} w_{k}+\sum_{j\in J} z_{j}
\nonumber
\\
&&
w_{k}+z_{j}\ge d_{j}^{k} \hspace{0.05cm} x_{j}^{*}
\hspace{1.9cm}
j\in J, k\in K
\nonumber
\\
&&
w_{k} \geq 0
\hspace{3.2cm}
k\in K
\nonumber
\\
&&
z_{j} \geq 0
\hspace{3.3cm}
j\in J \; ,
\nonumber
\end{eqnarray}

\noindent
Its dual problem is:
\begin{eqnarray}
\max
&&
\sum_{j \in J} \sum_{k \in K} d_j^k \hspace{0.05cm} x^{*}_j \hspace{0.1cm} y_{jk}
\nonumber
\\
&&
\sum_{j \in J} y_{jk} \leq \theta_k
\hspace{2.3cm}
k\in K
\nonumber
\\
&&
\sum_{j \in J} y_{jk} \leq 1
\hspace{2.4cm}
j \in J
\nonumber
\\
&&
y_{jk} \geq 0
\hspace{2.9cm}
j\in J, k\in K \; .
\nonumber
\end{eqnarray}

\noindent
This is exactly the maximum deviation problem DEV01 (see Section \ref{sec:compactMB})
defined for (BP) subject to cost deviation. So an optimal solution $y^{*}$
represents a maximum deviation scenario for $x^{*}$ and we have
$j \in J_k \Longleftrightarrow y^{*}_{jk} = 1$, for all $k\in K$.

Consider now a dual optimal binary solution $y^{*}$ and a primal feasible solution $(\bar{w},\bar{z})$.
By the well-known complementary slackness conditions, $(\bar{w},\bar{z})$ is optimal if and only if:
\begin{eqnarray}
y^{*}_{jk} \hspace{0.1cm}
(\bar{w}_{k} + \bar{z}_{j} - d_{j}^{k} \hspace{0.05cm} x_{j}^{*})
\hspace{0.1cm} &=& \hspace{0.1cm} 0
\hspace{1.5cm}
j\in J, k\in K
\nonumber
\\
\bar{w}_k \hspace{0.1cm}
\left(\theta_k - \sum_{j\in J} y_{kj}^{*}\right)
\hspace{0.1cm} &=& \hspace{0.1cm} 0
\hspace{1.5cm} k \in K
\nonumber
\\
\bar{z}_j \hspace{0.1cm}
\left(1 - \sum_{k \in K} y_{kj}^{*}\right)
\hspace{0.1cm} &=& \hspace{0.1cm} 0
\hspace{1.5cm} j \in J \; .
\nonumber
\end{eqnarray}

\noindent
Replacing $(\bar{w},\bar{z})$ with the inequalities defining any feasible solution $(w,z)$
we obtain:
\begin{eqnarray}
y^{*}_{jk} \hspace{0.1cm}
(w_{k} + z_{j} - d_{j}^{k} \hspace{0.05cm} x_{j}^{*})
\hspace{0.1cm} &=& \hspace{0.1cm} 0
\hspace{2.0cm}
j\in J, k\in K
\label{ineq_complSlack1}
\\
w_k \hspace{0.1cm}
\left(\theta_k - \sum_{j\in J} y_{kj}^{*}\right)
\hspace{0.1cm} &=& \hspace{0.1cm} 0
\hspace{2.0cm} k \in K
\label{ineq_complSlack2}
\\
z_j \hspace{0.1cm}
\left(1 - \sum_{k \in K} y_{kj}^{*}\right)
\hspace{0.1cm} &=& \hspace{0.1cm} 0
\hspace{2.0cm} j \in J
\label{ineq_complSlack3}
\\
w_{k}+z_{j}
\hspace{0.1cm} &\geq& \hspace{0.1cm}
d_{j}^{k} \hspace{0.05cm} x_{j}^{*}
\hspace{1.4cm}
j\in J, k\in K
\label{ineq_complSlack4}
\\
w_{k}
\hspace{0.1cm} &\geq& \hspace{0.1cm} 0
\hspace{2.0cm}
k\in K
\label{ineq_complSlack5}
\\
z_{j}
\hspace{0.1cm} &\geq& \hspace{0.1cm} 0
\hspace{2.0cm}
j\in J
\label{ineq_complSlack6}
\end{eqnarray}

\noindent
W.l.o.g. we can assume that the optimal solution $y^{*}$ is such that $\sum_{j \in J} y_{jk}^{*} = \theta_k$ and $\sum_{k \in K} y_{jk}^{*} = 1$. So we just need to consider (in)equalities (\ref{ineq_complSlack1}) and (\ref{ineq_complSlack4}-\ref{ineq_complSlack6}), as the remaining equalities are already satisfied by the choice of $y^{*}$. Hence, $(w,z)$ is optimal if and only if:
\begin{eqnarray}
w_{k} + z_{j}
\hspace{0.1cm} &=& \hspace{0.1cm}
d_{j}^{k} \hspace{0.05cm} x_{j}^{*}
\hspace{2.0cm}
k\in K, j \in J_k
\label{ineq_reducedComplSlack1}
\\
w_{k} + z_{j}
\hspace{0.1cm} &\geq& \hspace{0.1cm}
d_{j}^{k} \hspace{0.05cm} x_{j}^{*}
\hspace{2.0cm}
k\in K, j \in J \backslash J_k
\label{ineq_reducedComplSlack2}
\\
w_{k}
\hspace{0.1cm} &\geq& \hspace{0.1cm} 0
\hspace{2.55cm}
k\in K
\nonumber
\\
z_{j}
\hspace{0.1cm} &\geq& \hspace{0.1cm} 0
\hspace{2.55cm}
j\in J \; .
\nonumber
\end{eqnarray}

\noindent
Note that (\ref{ineq_reducedComplSlack1}),(\ref{ineq_reducedComplSlack2}) derive from (\ref{ineq_complSlack1}),(\ref{ineq_complSlack4}) as $y_{jk}^{*} = 1$ for all $k \in K, j \in J_k$.
\\
Since the sets $J_k$, $k \in K$ form a partition of $J$, by (\ref{ineq_reducedComplSlack1}) we can set $z_{j} = \hspace{0.1cm} d_{j}^{k} \hspace{0.05cm} x_{j}^{*} - w_{k}$ for all  $ k\in K, j \in J_k$. 
Yet, we must guarantee that $z_{j} \geq 0$ and so we must have $w_{k} \leq d_{j}^{k} \hspace{0.05cm} x_{j}^{*}$, for all $k\in K, j \in J_k$, which is equivalent to $w_{k} \leq \min_{j \in J_k} d_{j}^{k} \hspace{0.05cm} x_{j}^{*}$, for all $ k\in K$.
%

\noindent
Summarizing, for all $ k\in K$ we have the (in)equalities:
\begin{eqnarray}
&&
w_{k} \leq \min_{j \in J_k^{*}} d_{j}^{k} \hspace{0.05cm} x_{j}^{*}
\label{settingSetW_eq1}
\\
&&
z_{j}
= \hspace{0.1cm} d_{j}^{k} \hspace{0.05cm} x_{j}^{*} - w_{k}
\hspace{1.0cm}
j \in J_k \; .
\label{settingSetW_eq2}
\end{eqnarray}

\noindent
Consider now an inequality (\ref{ineq_reducedComplSlack2}) and denote by
$k'$ the deviation band of coefficient $j$ according to $x^{*}$ (i.e., $j \in J_{k'}$).
We can use equality (\ref{settingSetW_eq2}) in (\ref{ineq_reducedComplSlack2}) obtaining
$w_{k} + (d_{j}^{k'} \hspace{0.05cm} x_{j}^{*} - w_{k'}) \geq d_{j}^{k} \hspace{0.05cm} x_{j}^{*}$ and thus:
$$
w_{k} - w_{k'} \geq  d_{j}^{k} \hspace{0.05cm} x_{j}^{*} - d_{j}^{k'} \hspace{0.05cm} x_{j}^{*}
$$
We can repeat this reasoning for every inequality (\ref{ineq_reducedComplSlack2}) and reorganize the indices of the constraints obtaining the system:
$$
w_{k} - w_{k'} \geq  d_{j}^{k} \hspace{0.05cm} x_{j}^{*} - d_{j}^{k'} \hspace{0.05cm} x_{j}^{*}
\hspace{1.0cm}
k, k' \in K: k \neq k', \hspace{0.2cm} j \in J_{k'}
$$
which is equivalent to
$$
w_{k} - w_{k'} \geq
\max_{j \in J_{k'}} \{d_{j}^{k} \hspace{0.05cm} x_{j}^{*} - d_{j}^{k'} \hspace{0.05cm} x_{j}^{*}\}
\hspace{1.0cm}
k, k' \in K: k \neq k'
$$

\noindent
In total  we obtain: given a robust optimal solution $x^{*}$, $(w,z)$ is robust optimal as well if and only if:
\begin{eqnarray*}
w_{k} - w_{k'} &\geq&
\max_{j \in J_{k'}} \{d_{j}^{k} \hspace{0.05cm} x_{j}^{*} - d_{j}^{k'} \hspace{0.05cm} x_{j}^{*}\}
\hspace{1.75cm}
k, k' \in K: k \neq k'
\\
w_{k} &\leq& \min_{j \in J_k} d_{j}^{k} \hspace{0.05cm} x_{j}^{*}
\hspace{3.40cm}
k \in K
\\
w_{k} &\geq& 0
\hspace{4.70cm}
k \in K
\\
z_{j}
&=& \hspace{0.1cm} d_{j}^{k} \hspace{0.05cm} x_{j}^{*} - w_{k}
\hspace{3.15cm}
k \in K, j \in J_k \; .
\end{eqnarray*}
%
%
thus proving Claim 1.

\bigskip

\noindent
A crucial consequence of the proved Claim 1 is that we can compute an optimal robust solution $x_{j}^{*}$ to (ROB-BP) by solving $LP(b_{kk'},c_{k})$ for all possible combinations of values $b_{kk'}, c_{k}$:
\begin{eqnarray*}
b_{kk'} \in B_{kk'}
&=&
    \{d_{j}^{k} - d_{j}^{k'}, \forall j \in J\}
    \cup \{0\}
    \hspace{1.1cm}
\forall k,k' \in K: k \neq k'
\\
c_k \in C_k
&=&
    \{d_{j}^{k}, \hspace{0.1cm} \forall j \in J\}
    \cup \{0\}
    \hspace{1.9cm}
\forall k \in K
\end{eqnarray*}
Indeed
\begin{enumerate}
  \item if $\bar{w}$ is a feasible solution to $LP(b_{kk'},c_{k})$, then by Lemma \ref{lemma:fixedW} we solve (ROB-BP($\bar{w}$)) and obtain the optimal value Z($\bar{w}$) and an optimal solution $x(\bar{w})$ of (ROB-BP) for fixed $w = \bar{w}$;
  \\
  \item let $w(b_{kk'},c_{k})$ be a feasible solution to $LP(b_{kk'},c_{k})$. If $x(w(b_{kk'},c_{k}))$ corresponds with the smallest value $Z(w(b_{kk'},c_{k}))$ among all the combinations $b_{kk'} \in B_{kk'}, c_k \in C_k$, then $x(w(b_{kk'},c_{k}))$ is a robust optimal solution to (ROB-BP).
\end{enumerate}

\noindent
Algorithm 1.1 formally translates these observations by looking for the combination $b_{kk'} \in B_{kk'},c_k \in C_k$ that is associated with the smallest value $Z(w(b_{kk'},c_{k}))$.
Since the number of combinations is equal to $(n+1)^{{k}^{2}}$, we must solve at most $(n+1)^{{k}^{2}}$ problem (ROB-BP($w$)) (the problem must not be solved when the corresponding $LP(\cdot)$ does not admit a feasible solution $w$).

Note that the tractability of the entire algorithm depends on the tractability of the nominal optimization problem (BP). If there exists a polynomial algorithm for (BP), then Algorithm 1.1 is polynomial as well.

This completes the proof.
\
\qed
\end{proof}

\begin{remark} \label{reamrk:noBands}
We note that Algorithm 1.1 can be transferred to a robust $\alpha$- approximation
algorithm if instead of solving (ROB-BP($w$)) to optimality we compute
an $\alpha$-approximation of the nominal problem.
\end{remark}

%
%

\section{Defining a probability bound of constraint violation
\label{sec:prob_bound}}

In this section, we analyze the probability that a robust optimal solution $x^{*}$ becomes infeasible: $x^{*}$ is indeed completely protected against deviations captured by the multi-band scenario set, but it may still become infeasible under deviations that lie outside the set.
In contrast to the bounds defined in \cite{BeNe00} and \cite{BeSi04}, which use a priori information about the not completely known distribution of the deviations (specifically, they assume symmetric distributions with symmetric range), we define a bound that exploits historical data available on the uncertain coefficients. Our basic assumption is that for each coefficient $a_{ij}$, we have $W > 0$ samples $a_{ij}^{1}, a_{ij}^{2}, \ldots, a_{ij}^{W}$, with each sample lying in the overall deviation range, i.e. $a_{ij}^{\sigma} \in [\bar{a}_{ij} - d_{ij}^{K-}, \bar{a}_{ij} + d_{ij}^{K+}]$, $\sigma = 1, \ldots, W$. This is a very reasonable assumption in real-world problems, where data about the past behavior of the uncertainty are commonly available. Through this samples, we can define the \emph{sample mean} $\mu_{ij} = 1/W \sum_{\sigma = 1}^{W} a_{ij}^{\sigma}$ of $a_{ij}$ and define a new typology of probabilistic bound, as formalized in the proposition that follows. We note that our bound holds with a bounded probability, since it is defined using the well-known Hoeffding's Inequality \cite{Ho63}: Hoeffding's Inequality bounds the probability that the difference between the sample mean and the actual mean of a distribution exceeds a fixed threshold.

\begin{proposition}
Consider the robust optimization problem (Rob-MILP) under a multi-band uncertainty set defined as in Section \ref{sec:intro} and let $x^{*}$ be a robust optimal solution to (Rob-MILP). Suppose that the uncertain coefficients $a_{ij}$ are independent random variables and that for each  $a_{ij}$, a number $W > 0$ of samples $a_{ij}^{\sigma} \in [\bar{a}_{ij} - d_{ij}^{K-}, \bar{a}_{ij} + d_{ij}^{K+}]$, with $\sigma = 1, \ldots, W$ is available and denote by $\mu_{ij}$ the corresponding sample mean.
\\
Given a constraint $i \in I$ and a constant $t \geq 0$, the probability $P^{V}_i(x^{*})$ that the robust optimal solution $x^{*}$ violates $i$ (i.e., $P^{V}_i(x^{*}) = P [a_i' x^{*} > b_i])$ is such that:
\begin{equation}\label{proposition_probBound}
P^{V}_i(x^{*})
\leq
\exp \left(
- t \hspace{0.05cm} b_i
\hspace{0.1cm}
+ \sum_{j \in J} \hspace{0.05cm} \ln \hspace{0.05cm} B_{ij}[t,x^{*}]
\right)
\; ,
\end{equation}
where $
B_{ij}[t,x^{*}]
=
\frac{1}{d_{ij}^{K+} + d_{ij}^{K-}}
\hspace{0.1cm}
\cdot
$
\begin{eqnarray*}
&&
\cdot
\left\{
    \left[
        \bar{a}_{ij} + d_{ij}^{K+}
        - \left(
            \mu_{ij} + x_j^{*} \left(d_{ij}^{K+} + d_{ij}^{K-}\right) \sqrt{\frac{1}{2 W \ln \beta_{ij}}}
            \hspace{0.05cm}
        \right)
    \right]
        \exp \left(
                t \hspace{0.05cm} x_j^{*} (\bar{a}_{ij} - d_{ij}^{K-})
             \right)
     -
\right.
\\
&&
\\
&&
\left.
-
    \left[
        \bar{a}_{ij} - d_{ij}^{K-}
        - \left(
            \mu_{ij} + x_j^{*} \left(d_{ij}^{K+} + d_{ij}^{K-}\right) \sqrt{\frac{1}{2 W \ln \beta_{ij}}}
        \right)
    \right]
        \exp \left(
                t \hspace{0.05cm} x_j^{*} (\bar{a}_{ij} + d_{ij}^{K+})
             \right)
\right\} \;.
%
%
%
%
\end{eqnarray*}

\noindent
Moreover, for fixed $\tau > 0$, the bound (\ref{proposition_probBound}) holds with probability greater than or equal to
$$
\prod_{j \in j} (1 - \beta_{ij})
=
\prod_{j \in j} \left[ 1 -
\exp \left(
        -\frac{2 \hspace{0.05cm} \tau^{2} \hspace{0.05cm} W}{x_j^{*^{4}} (d_{ij}^{K+} + d_{ij}^{K-})^{2}}
    \right) \right] \; .
$$
\end{proposition}

\smallskip

\begin{proof}
In order to prove the result, we will use the well-known Markov's Inequality and Hoeffding's Inequality \cite{Ho63}.
Given $x^{*}$ and $t > 0$, the probability of violation of constraint $i$ is:
\begin{eqnarray}
\label{1stStepBound}
P^{V}(x^{*})
\hspace{0.1cm}
=
\hspace{0.1cm}
P
  \left[
  \sum_{j \in J} a_{ij} \hspace{0.1cm} x_{j}^{*} > b_i
  \right]
\hspace{0.1cm}
\stackrel{(1b)}{\leq}
\hspace{0.1cm}
\exp \left(
        - t \hspace{0.05cm} b_i
     \right)
\hspace{0.1cm}
\prod_{j \in J}
     E \left[
        \exp \left(
                t \hspace{0.05cm} a_{ij} \hspace{0.05cm} x_{j}^{*}
            \right)
     \right] \; ,
\end{eqnarray}

\noindent
where inequality (1b) derives from Markov's Inequality and the independence of $a_{ij}$.
We are then interested in bounding the moment generating function $E \left[
        \exp \left(
                t \hspace{0.05cm} a_{ij} \hspace{0.05cm} x_{j}^{*}
            \right)
     \right]$ in the r.h.s.~of (1b), using the available samples $a_{ij}^{\sigma}$, $\sigma = 1, \ldots, W$.
To this end, we can consider the following formula, that exploits the convexity of the exponential function and holds for given $\gamma > 0$ for a bounded random variable $V: l \leq V \leq u$ \cite{Ho63}:
\begin{eqnarray}
\label{probBound_boundMoment}
E \left[
    e^{\gamma \hspace{0.05cm} V}
  \right]
\hspace{0.1cm}
&\leq&
\frac{1}{u - l}
\hspace{0.1cm}
\left[
\hspace{0.1cm}
\left(
u - E[V] \hspace{0.05cm}
\right)
\hspace{0.1cm}
e^{\gamma \hspace{0.05cm} l}
\hspace{0.1cm}
+
\hspace{0.05cm}
\left(
E[V] - l
\right)
\hspace{0.1cm}
e^{\gamma \hspace{0.05cm} u}
\hspace{0.1cm}
\right]
    \; .
\end{eqnarray}
To use this formula, since we do not know the real mean $E[a_{ij}]$, in the following passages we derive a bound on $E[a_{ij}]$.

As first step, for each coefficient $a_{ij}$, we use Hoeffding's Inequality to bound the probability that the difference between the sample mean and the actual mean is above a value $\tau > 0$, namely:
\begin{eqnarray*}
P
  \left[
    x_{j}^{*} E[a_{ij}] - x_{j}^{*} \mu_{ij} \geq \tau
  \right]
&=&
P
  \left[
    E[a_{ij}] - \mu_{ij} \geq \frac{\tau}{x_{j}^{*}}
  \right]
  \hspace{3.9cm} \leq
  \\
&\leq&
\exp \left(
        -\frac{2 \hspace{0.05cm} \left(\frac{\tau}{x_j^{*}}\right)^{2} \hspace{0.05cm} W^{2}}
        {W \left[
                x_j^{*} (\bar{a}_{ij} + d_{ij}^{K+}) - x_j^{*} (\bar{a}_{ij} - d_{ij}^{K-})
           \right]^{2}}
    \right)
    \hspace{0.2cm} =
\\
&=&
\exp \left(
        -\frac{2 \hspace{0.05cm} \tau^{2} \hspace{0.05cm} W}
        {x_j^{*^{4}} (d_{ij}^{K+} + d_{ij}^{K-})^{2}}
    \right)
    \; .
\end{eqnarray*}

\noindent
By letting $\beta_{ij} =\exp \left(
        -\frac{2 \hspace{0.05cm} \tau^{2} \hspace{0.05cm} W}
        {x_j^{*^{4}} (d_{ij}^{K+} + d_{ij}^{K-})^{2}}
    \right)$
and by reorganizing the equality through simple algebra operations we get
%
$
\tau =
x_j^{*^{2}}
(d_{ij}^{K+} + d_{ij}^{K-})
\sqrt{\frac{1}{2 W \ln \beta_{ij}}}
$
%
%
and finally:
\begin{eqnarray}
\label{probBound_boundMean1}
P
  \left[
    E[a_{ij}] - \mu_{ij}
    \hspace{0.1cm}
    \geq
    \hspace{0.1cm}
    x_j^{*}
    \hspace{0.1cm}
        (d_{ij}^{K+} + d_{ij}^{K-})
        \hspace{0.1cm}
        \sqrt{\frac{1}{2 W \ln \beta_{ij}}}
    \hspace{0.2cm}
  \right]
  \hspace{0.2cm}
  \leq
  \hspace{0.2cm}
  \beta_{ij}
    \; .
\end{eqnarray}

\noindent
From inequality (\ref{probBound_boundMean1}), we can derive the following bound for the (unknown) actual mean of $a_{ij}$:
\begin{eqnarray}
\label{probBound_boundMean2}
    E[a_{ij}]
    \hspace{0.1cm}
    \leq
    \mu_{ij} +
    \hspace{0.1cm}
    x_j^{*}
    \hspace{0.1cm}
        (d_{ij}^{K+} + d_{ij}^{K-})
        \hspace{0.1cm}
        \sqrt{\frac{1}{2 W \ln \beta_{ij}}}
    \; .
\end{eqnarray}
that holds with probability at least $1 - \beta_{ij}$.

We can then use the bound (\ref{probBound_boundMean2}) on the actual mean in the bound (\ref{probBound_boundMoment}) on the moment generating function, to finally reach our objective, namely defining a bound on
$E \left[
        \exp \left(
                t \hspace{0.05cm} a_{ij} \hspace{0.05cm} x_{j}^{*}
            \right)
\right]
$. In our case, the adaptation of (\ref{probBound_boundMoment}) is:
\begin{eqnarray}
E \left[
        \exp \left(
                t \hspace{0.05cm} a_{ij} \hspace{0.05cm} x_{j}^{*}
            \right)
\right]
&\leq&
\frac{1}{(d_{ij}^{K+} + d_{ij}^{K-})}
\cdot
\nonumber
\\
&&
\cdot
\left\{
    \left(
        \bar{a}_{ij} + d_{ij}^{K+}
        - E[a_{ij}]
    \right)
    \hspace{0.05cm}
        \exp \left(
                t \hspace{0.05cm} x_j^{*} (\bar{a}_{ij} - d_{ij}^{K-})
             \right)
     -
\right.
\nonumber
\\
&&
\left.
-
    \left(
        \bar{a}_{ij} - d_{ij}^{K-} - E[a_{ij}]
    \right)
    \hspace{0.05cm}
        \exp \left(
                t \hspace{0.05cm} x_j^{*} (\bar{a}_{ij} + d_{ij}^{K+})
             \right)
\right\} .
\label{probBound_finalBound}
\end{eqnarray}

\noindent
where we note that we have not yet substituted $E[a_{ij}]$.

Finally, we substitute the bound (\ref{probBound_finalBound}) in (\ref{1stStepBound}), thus obtaining the bound $B_{ij}[t,x^{*}]$ of the statement of the Proposition, that holds with probability greater than or equal to $\prod_{j \in j} (1 - \beta_{ij})$. We then have:
\begin{eqnarray*}
\exp \left(
        - t \hspace{0.05cm} b_i
     \right)
\hspace{0.1cm}
\prod_{j \in J}
     E \left[
        \exp \left(
                t \hspace{0.05cm} a_{ij} \hspace{0.05cm} x_{j}^{*}
            \right)
     \right]
&\leq&
\exp \left(
        - t \hspace{0.05cm} b_i
     \right)
\prod_{j \in J}
B_{ij}[t,x^{*}]
=
\\
&=&
\exp \left(
- t \hspace{0.05cm} b_i
\hspace{0.1cm}
+ \sum_{j \in J} \hspace{0.05cm} \ln \hspace{0.05cm} B_{ij}[t,x^{*}]
\right)
\; ,
\end{eqnarray*}
ending the proof.
\qed
\end{proof}

\section{Conclusions and Future Work}  \label{sec:end}

In this work, we presented new theoretical results abound multi-band uncertainty in Robust Optimization. Surprisingly, this natural refinement of the classical single band model by Bertsimas and Sim has attracted very little attention and we have thus worked
on filling the existent knowledge gap. Our ongoing research is currently focused on refining the cutting plane method and intensifying the computational experiments to other relevant real-world problems, considering realistic instances defined in collaboration with our industrial partners.

\end{document}